\documentclass[12pt]{article}
\usepackage{amssymb}
\usepackage{amsmath}
\usepackage{amsthm}
\usepackage{latexsym}
\usepackage{mathdots}
\usepackage{tikz}
\usepackage{graphicx}
\usetikzlibrary{decorations}
\usetikzlibrary{decorations.pathmorphing}
\usetikzlibrary{shapes,snakes}
\DeclareSymbolFont{bbold}{U}{bbold}{m}{n}
\DeclareSymbolFontAlphabet{\mathbbold}{bbold}
\newtheorem{theo}{Theorem}[section]
\newtheorem{prop}[theo]{Proposition}
\newtheorem{cor}[theo]{Corollary}
\newtheorem{lemma}[theo]{Lemma}
\theoremstyle{definition}
\newtheorem{defi}[theo]{Definition}
\newtheorem{exa}[theo]{Example}
\newtheorem{rem}[theo]{Remark}

\newtheorem{que}[theo]{Question}
\numberwithin{equation}{section}

\newcommand{\Ds}{\displaystyle}
\newcommand{\Ts}{\textstyle}
\newcommand{\N}{{\mathbb N}}
\newcommand{\F}{{\mathbb F}}
\newcommand{\Z}{{\mathbb Z}}

\newcommand{\C}{{\mathbb C}}

\newcommand{\cA}{{\mathcal A}}
\newcommand{\cC}{{\mathcal C}}

\newcommand{\cB}{{\mathcal B}}

\newcommand{\cI}{{\mathcal I}}
\newcommand{\cJ}{{\mathcal J}}

\newcommand{\cS}{{\mathcal S}}
\newcommand{\cN}{{\mathcal N}}

\newcommand{\ideal}[1]{\mbox{$\langle{#1}\rangle$}}
\newcommand{\dd}{\mbox{\rm d}}
\newcommand{\rb}{\mbox{\rm rb}}

\newcommand{\wt}{\mbox{\rm wt}}
\newcommand{\girth}{\mbox{\rm girth}}

\newcommand{\mbone}{{\mathbbold{1}}}

\newcounter{alp}
\newcounter{ara}
\newcounter{rom}
\newenvironment{romanlist}{\begin{list}{\roman{rom})\hfill}{\usecounter{rom}
     \topsep0ex \labelwidth.7cm \leftmargin.7cm \labelsep0cm
     \rightmargin0cm \parsep0ex \itemsep.4ex
     \partopsep1ex}}{\end{list}}
\newenvironment{alphalist}{\begin{list}{(\alph{alp})\hfill}{\usecounter{alp}
     \topsep0ex \labelwidth.7cm \leftmargin.7cm \labelsep0cm
     \rightmargin0cm \parsep0ex \itemsep.5ex
     \partopsep0ex}}{\end{list}}
\newenvironment{arabiclist}{\begin{list}{\arabic{ara})\hfill}{\usecounter{ara}
     \topsep0ex \labelwidth.6cm \leftmargin.6cm \labelsep0cm
     \rightmargin0cm \parsep0ex \itemsep.8ex
     \partopsep0ex}}{\end{list}}

\begin{document}
\title{On Robust Colorings of Hamming-Distance Graphs}
\date{\today}
\author{Isaiah Harney$^\ast$ and Heide Gluesing-Luerssen\footnote{HGL was partially supported by the National
Science Foundation Grant DMS-1210061.
Both authors are with the Department of Mathematics, University of Kentucky, Lexington KY 40506-0027, USA;
\{isaiah.harney,heide.gl\}@uky.edu.}}

\maketitle

{\bf Abstract:}
$H_q(n,d)$ is defined as the graph with vertex set $\Z_q^n$ and where two vertices are adjacent if their
Hamming distance is at least~$d$.
The chromatic number of these graphs is presented for various sets of parameters $(q,n,d)$.
For the $4$-colorings of the graphs $H_2(n,n-1)$ a notion of robustness is introduced.
It is based on the tolerance of swapping colors along an edge without destroying properness of the coloring.
An explicit description of the maximally robust $4$-colorings of $H_2(n,n-1)$ is presented.



\section{Introduction}
In this paper we study vertex colorings of graphs based on the Hamming distance.
Precisely, fix $q\geq2$ and $d\leq n$ and consider the graph with vertex set $\Z_q^n$ and
where two vertices are adjacent if their Hamming distance is at least~$d$.
We denote these graphs by $H_q(n,d)$ and call them Hamming-distance graphs or shortly Hamming graph
(the latter terminology, however, is used for a variety of graphs in the literature and therefore needs to be taken with care).
Note that the cliques in $H_q(n,d)$ are exactly the $q$-ary block codes of length~$n$ and Hamming distance at least~$d$, whereas the
independent sets are the anticodes with maximum distance $d-1$.

For general parameters, $H_q(n,d)$ is not a graph product of any standard sort (weak, strong, cartesian etc.) and therefore the
abundant literature on the chromatic number and colorings of graph products (see, e.g., the survey~\cite{Kl96} by Klav\u{z}ar
and the references therein) does not apply.

However, for $d=n$ the graph $H_q(n,n)$ is the $n$-th weak power of the complete graph~$K_q$ on~$q$ vertices, and its
chromatic number as well as all minimal colorings are known.
Precisely, the chromatic number of this graph is~$q$ and (for $q>2$) each $q$-coloring is given by a
coordinate coloring, that is, by simply assigning the value of a fixed coordinate as the color; see~\cite{GrLo74} by Greenwell/Lov\'{a}sz
as well as~\cite{ADFS04} by Alon et al., where -- among many other results -- the authors present an alternative
proof based on Fourier analytical methods.
By applying a certain relation between a function and its Fourier transform to the indicator function of a maximal
independent set they derive the exact structure of that set.

We will begin with discussing the chromatic number of $H_q(n,d)$.
This discussion is closely related to~\cite{RGSS07} by Rouayheb et al., where the authors link bounds on codes to properties
of the Hamming graphs.
In~\cite{JaMa06}, Jamison/Matthews study the chromatic number of the graph complement of $H_q(n,d)$ for various sets of parameters.
However, their results combined with the well-known Nordhaus-Gaddum relations between the chromatic numbers of a graph and its complement do
not lead to any tighter bounds than those presented in this paper.

As for the chromatic number, it is straightforward to generalize the idea of coordinate colorings
of $H_q(n,n)$ in~\cite{ADFS04} to $H_q(n,d)$.
This trivially leads to the upper bound $q^{n-d+1}$ for the chromatic number of $H_q(n,d)$.
As a consequence, the chromatic number of $H_2(n,n-1)$ is~$4$ because Payan~\cite{Pa92} showed that the chromatic number
of a nonbipartite cube-like graph is at least~$4$.
Furthermore, with the aid of the Erd\H{o}s-Ko-Rado Theorem for integer sequences~\cite{AhKh98,FrTo99,Kl66} it can be shown that
for various other sets of parameters,
$q^{n-d+1}$ is in fact the chromatic number; see Corollary~\ref{C-ConjCases} for details.
The general case, however, remains open (to the best of our knowledge).
This includes $H_2(n,d)$ for most pairs $(n,d)$ with $d<n-1$.

In the main part of the paper we will focus on the graphs $H_2(n,n-1)$, which have chromatic number~$4$.
We show that -- different from $H_q(n,n)$ -- these graphs enjoy $4$-colorings that are not coordinate colorings.
Even more, they also allow uneven $4$-colorings (i.e., the color classes do not have the same size).
This irregular structure of maximal independent sets and minimal colorings suggests that the Fourier analytical methods of~\cite{ADFS04}
are unlikely to be of any help in order to determine the exact chromatic number for general $H_q(n,d)$.

The abundance of genuinely different $4$-colorings of $H_2(n,n-1)$ gives rise to the question of whether the colorings can be ranked in a certain way.
We approach this question by introducing a measure of robustness, which  is based on
the number of transition edges.
We call an edge a transition edge of a given coloring if swapping the colors of the two incident vertices leads to a new proper coloring.
The robustness is simply the fraction of transition edges among all edges.
We show that if~$n$ is odd the coordinate colorings are exactly the maximally robust $4$-colorings of $H_2(n,n-1)$, while for even~$n$
there is an additional (natural) type of coloring with maximal robustness.
Furthermore, we show that a $4$-coloring is maximally robust if and only if the transition edges tile the graph in disjoint $4$-cycles.

Throughout the paper we use the following terminology and notation.
For any graph~$G$ we denote by $V(G)$ its vertex set and by $E(G)$ its edge set.
Thus $E(G)\subseteq V(G)\times V(G)$.
All graphs have neither loops nor multiple edges.
Furthermore, graphs are undirected, and thus
$(x,y)\in E(G)\Longleftrightarrow (y,x)\in E(G)$.
We also use the notation $\sim$ for adjacency, thus $x\sim y\Longleftrightarrow (x,y)\in E(G)$.
Let $\cN(x):=\{y\in V(G)\mid y\sim x\}$ denote the set of neighbors of vertex $x\in V(G)$.
For any graph~$G$ we define the following invariants:
$\alpha(G)$ denotes its independence number, that is, the maximal size of an independent set;
$\omega(G)$ is the clique number, that is, the maximal size of a clique in~$G$;
$\girth(G)$ denotes the girth, that is, the length of a shortest cycle.
A proper $k$-coloring is a vertex coloring such that adjacent vertices have different colors.
We only consider proper colorings and usually omit the qualifier `proper' unless we explicitly discuss properness.
We denote by~$\chi(G)$ the chromatic number, that is, the minimal~$k$ such that~$G$ has a proper $k$-coloring.
Any $\chi(G)$-coloring is called a \emph{minimal coloring}.
A coloring is called \emph{even} if all color classes (the sets of vertices with the same color) have the same cardinality.
Finally, for any $x,\,y\in V(G)$ we define $\dd_G(x,y)$ as the graph distance between~$x$ and~$y$, that is, the length of a
shortest path from~$x$ to~$y$.
The vector with all entries equal to~$1$ is denoted by~$\mbone$.
Its length will be clear from the context.
We set $[n]:=\{1,\ldots,n\}$.

\section{Basic Properties of the Hamming-Distance Graphs}
Let~$q,\,n\in\N_{\geq2}$ and let~$A$ be a set of cardinality~$q$.
For any two vectors $x,y$ in $A^n$ we denote by $\dd(x,y):=|\{i\mid x_i\neq y_i\}|$ the Hamming distance between~$x$ and~$y$ and
by $\wt(x):=\dd(x,0)$ the Hamming weight of~$x$.
Recall that a \emph{$q$-ary block code} of length~$n$ and Hamming distance~$d$ over the alphabet~$A$ is a subset, say~$\cC$, of~$A^n$
such that $\dd(x,y)\geq d$ for all distinct elements $x,y\in\cC$.
The block code is called \emph{linear} if~$A$ is a field of size~$q$ (thus~$q$ is a prime power) and $\cC$ is a subspace of the vector space~$A^n$.

Throughout this paper we choose $A=\Z_q:=\Z/q\Z$.
This is solely for convenience and any alphabet of cardinality~$q$ would work.
However, some arguments are simplified by using the group structure of~$\Z_q$, i.e., computing modulo~$q$.
In the very few cases where we consider linearity we will switch to a field of size~$q$.

From now on we assume $q,d,n\in\N$ such that
\begin{equation}\label{e-qnd}
   q,\,n\geq2,\quad 0<d\leq n,\quad (q,d)\neq(2,n).
\end{equation}

The following graph has been introduced and briefly discussed by Sloane~\cite{Sl89} and Rouayheb et al.~\cite{RGSS07}.

\begin{defi}\label{D-Ham}
Let $q,\,d,\,n\in\N$  be as in~\eqref{e-qnd}.
The \emph{Hamming-distance graph}, or simply \emph{Hamming graph}, $H_q(n,d)$ is defined as the graph with vertex set
$V:=\Z_q^n$ and edge set $E:=\{(x,y)\mid x,\,y\in\Z_q^n,\,\dd(x,y)\geq d\}$.
\end{defi}
We point out that in the literature the terminology ``Hamming graph'' sometimes refers to other graphs,
among them for instance our special case~$H_q(n,n)$.
In this paper the name exclusively refers to the graphs in the above definition.

Note that the graph $H_q(n,n)$ is the $n$-th weak power of the complete graph $K_q$ on~$q$ vertices; see also \cite[p.~914]{ADFS04}.

By definition of the Hamming graphs, the cliques form $q$-ary block codes of length~$n$ and Hamming distance
at least~$d$ and thus the clique number $\omega(H_q(n,d))$ is the maximal cardinality of such a code.
On the other hand, the independent sets are the anticodes in~$\Z_q^n$ of distance~$d-1$, that is, any two elements in such a set
have distance at most~$d-1$.

Let us briefly look at the case $(q,d)=(2,n)$, which is excluded in~\eqref{e-qnd}.
This case is genuinely different from the general one because
the graph $H_2(n,n)$ is disconnected with $2^{n-1}$ components each one consisting of one edge.
Indeed, every~$x\in\Z_2^n$ has only one neighbor, namely $x+\mbone$
(recall that $n\geq2$).
Thus $H_2(n,n)$ is entirely uninteresting and excluded from our considerations.

We list some basic properties of the Hamming graphs.
\begin{prop}\label{P-basics}\
\begin{alphalist}
\item $H_q(n,d)$ is regular of degree $\sum_{i=0}^{n-d}{n \choose i}(q-1)^{n-i}$.
\item $H_q(n,d)$ is vertex-transitive.
\item $H_q(n,1)$ is the complete graph on $q^n$ vertices.
\item The Hamming graph $H_q(n,d)$ is the undirected Cayley graph of the group
      $(\Z_q^n,+)$ and the subset $\cS_d=\{v\in\Z_q^n\mid \wt(v)\geq d\}$, which generates the group.
\item $H_q(n,d)$ is connected.
\end{alphalist}
\end{prop}
\begin{proof}
(a) For any vertex~$x\in\Z_q^n$ the neighbors are exactly those vectors that agree with~$x$ in at most~$n-d$ coordinates.
There are clearly ${n \choose i}(q-1)^{n-i}$ vectors that share exactly~$i$ coordinates with~$x$. Summing over all possible~$i$
yields the desired formula.
\\
(b) For any two vertices $x,\,y\in\Z_q^n$ the map $\Z_q^n\longrightarrow \Z_q^n,\ v\longmapsto v+(x-y)$ defines a graph
automorphism that maps~$y$ to~$x$.
\\
(c) is trivial.
\\
(d) Since $(q,d)\neq(2,n)$, the set~$\cS_d$ indeed generates the group $\Z_q^n$. Even more, the vectors
$(1,1,\ldots,1),(2,1,\ldots,1),(1,2,1\ldots,1),\ldots,(1,\ldots,1,2,1)$ form a basis of the $\Z_q$-module~$\Z_q^n$
and are in~$\cS_d$ for any admissible~$d$.
By the very definition of the Hamming graph, $x,\,y\in\Z_q^n$ are adjacent in $H_q(n,d)$ iff $y=x+v$ for some $v\in\cS_d$.
This shows that the graph is the stated Cayley graph.
\\
(e) follows from~(d).
\end{proof}

\begin{prop}\label{P-cycles}
Let $G:=H_q(n,d)$.
\begin{alphalist}
\item Let $q\geq3$. Then $\dd_G(x,y)\leq 2$ for all $x,y\in V(G)$ and $\girth(G)=3$.
\item Let $q=2$ and $d\leq\frac{2n}{3}$. Then $\girth(G)=3$.
\item Let $q=2$ and $\frac{2n}{3}<d<n$. Then $\girth(G)=4$.
      Moreover,~$G$ contains a cycle of odd length.
\end{alphalist}
\end{prop}
\begin{proof}
(a) Let $x,\,y\in\Z_q^n$ be any distinct vertices. Since $q\geq3$ there exist $z_i\in\Z_q$ such that $x_i\neq z_i\neq y_i$ for all $i=1,\ldots,n$.
Thus $(x,z),\,(z,y)\in E(G)$, and we have a path of length~$2$.
The girth follows from the cycle $0,\,\mbone,\,(2,\ldots,2)$. 
\\
(b) Consider the vertices $0,\,x,\,y$, where $x=(1,\ldots,1,0,\ldots,0)$ and $y=(0,\ldots,0,1\ldots,1)$ with both having
exactly~$d$ entries equal to~$1$.
If $n\geq 2d$, then $\dd(x,y)=2d$ and thus~$0,\,x,\,y$ form a cycle of length~$3$.
Let $n<2d$.
Then~$x$ and~$y$ overlap in exactly $2d-n$ entries (which are equal to~$1$), and therefore $\dd(x,y)=2n-2d\geq 3d-2d=d$.
Again, the three vertices form a cycle of length~$3$.
\\
(c)
First of all, the vertices $0,\,\mbone,\,(1,0,\ldots,0),\,(0,1,\ldots,1)$ form a cycle of length~$4$ in the Hamming graph.
By vertex-transitivity it now suffices to show that~$0$ is not contained in a cycle of length~$3$.
Suppose $x,y\in\Z_q^n$ are adjacent to~$0$. Then $\wt(x),\,\wt(y)\geq d$.
Moreover,  $x_i=1=y_i$ for at least $2d-n>\frac{n}{3}$ positions and therefore $\dd(x,y)<n-\frac{n}{3}=\frac{2n}{3}<d$.
Hence $x,\,y$ are not adjacent, and this shows that the graph has no cycle of length~$3$.
A cycle of odd length is obtained as follows.
For $i=1,\ldots,n$ set $f_i:=(1,\ldots,1,0,1,\ldots,1)$, where~$0$ is at position~$i$.
If~$n$ is odd the vertices $0,f_1,f_1+f_2,\ldots,\sum_{i=1}^n f_i=0$ form a cycle of length~$n$.
If $n$ is even the cycle
$0,\,f_1,\,f_1+f_2,\ldots,\sum_{i=1}^nf_i=\mbone,\,0$ has length~$n+1$.
\end{proof}

\section{On the Chromatic Number of the Hamming Graph}\label{S-chi}
There is an obvious way to color any Hamming graph, and it will play a central role in our investigation of  colorings.
\begin{defi}\label{D-CoordCol}
Consider the Hamming graph $H_q(n,d)$.
Fix $1\leq i_1<i_2<\ldots<i_{n-d+1}\leq n$.
The $(i_1,\ldots,i_{n-d+1})$-\emph{coordinate coloring} of $H_q(n,d)$ is defined as
\[
    K:\Z_q^n\longrightarrow \Z_q^{n-d+1},\ (x_1,\ldots,x_n)\longmapsto (x_{i_1},\ldots,x_{i_{n-d+1}}).
\]
It is a proper $q^{n-d+1}$-coloring of $H_q(n,d)$.
\end{defi}
The properness of~$K$ follows from the fact that adjacent vertices have Hamming distance at least~$d$ and therefore
cannot agree on any $n-d+1$ coordinates.
As a consequence,
\begin{equation}\label{e-chi}
  \alpha(H_q(n,d))\geq q^{d-1}\ \text{ and }\ \chi(H_q(n,d))\leq q^{n-d+1}.
\end{equation}

\begin{que}\label{Conj-chi}
Is $\chi(H_q(n,d))= q^{n-d+1}$?
\end{que}

In this section we discuss existing results and several methods to determine the chromatic number explicitly, thereby answering the
question in the affirmative for various sets of parameters.

First of all, we trivially have $\chi(H_q(n,1))=q^n$  because $H_q(n,1)$ is the complete graph on~$q^n$ vertices.
Next, Payan~\cite{Pa92} proved the surprising fact that no nonbipartite cube-like graph can have chromatic number less than~$4$.
Since the binary Hamming graph $H_2(n,d)$ is the cube-like graph $Q_n[d,d+1,\ldots,n]$ we obtain from~\eqref{e-chi}
\begin{theo}\label{T-H2nnm1}
$\chi(H_2(n,n-1))=4$.
\end{theo}

Now we turn to more general cases.
For a proper $r$-coloring of $H_q(n,d)$ the color classes, that is, the set of all vertices with the same color, are independent sets,
and thus the existence of a proper $r$-coloring implies the existence of an independent set of cardinality at least $\lceil{q^n/r}\rceil$.
This leads to the sufficient condition
\begin{equation}\label{e-IndSetChi}
  \alpha(H_q(n,d))<\Big\lceil{\frac{q^n}{q^{n-d+1}-1}}\Big\rceil\Longrightarrow \chi(H_q(n,d))= q^{n-d+1}.
\end{equation}
Furthermore,
\begin{equation}\label{e-EvenCol}
  \alpha(H_q(n,d))=q^{d-1}\ \Longrightarrow\ \text{ each $q^{n-d+1}$-coloring of $H_q(n,d)$ is even.}
\end{equation}
One cannot expect~\eqref{e-IndSetChi} to be very fruitful in determining the chromatic number because it is only a sufficient condition.
In other words, the existence of one independent set with at least $\lceil q^n/(q^{n-d+1}-1)\rceil$ elements does not imply that
$\chi(H_q(n,d))< q^{n-d+1}$ because the latter requires the existence of multiple independent sets
(possibly of varying cardinality).
Yet, some cases can indeed be settled this way as the independence number of the Hamming graph is known for most sets of parameters.
This result is known as the Erd\H{o}s-Ko-Rado Theorem for integer sequences
and has been proven by Kleitman~\cite{Kl66} for the binary case
and Ahlswede/Khachatrian~\cite{AhKh98} as well as Frankl/Tokushige~\cite{FrTo99} for the general case.
In our notation it states the following.

\begin{theo}\label{T-alpha}\
\begin{alphalist}
\item \cite{Kl66} or~\cite[Thm.~K1]{AhKh98} Let $q=2$. Then
      \[
         \alpha(H_2(n,d))=\left\{\begin{array}{cl} {\Ds\sum_{i=0}^{\frac{d-1}{2}}{n\choose i}},&\text{ if $d$ is odd,}\\[1.5ex]
                                                   {\Ds 2\sum_{i=0}^{\frac{d-2}{2}}{n-1\choose i}},&\text{ if $d$ is even.}
                          \end{array}\right.
      \]
\item \cite[Thm.~2 and pp.~57-58, Cor.~1]{FrTo99} Let $q\geq3$. Set $r:=\lfloor{\frac{n-d}{q-2}}\rfloor$. Then
      \[
        \alpha(H_q(n,d))\leq\Big\lfloor q^{d-1-2r}\sum_{i=0}^r{n-d+1+2r\choose i}(q-1)^i\Big\rfloor,
      \]
      with equality if $d\geq 2r+1$. In particular, $\alpha(H_q(n,d))=q^{d-1}$ for $d\geq n-q+2$ (and $d>0$).
\end{alphalist}
\end{theo}

For $d=n$ the very last part appears also in \cite[Claim 4.1]{ADFS04}.
In Theorem~1.1 of the same paper it is then shown that each maximal independent set of $H_q(n,n)$ (where $q\geq3$) is of the form
$\cI_{i,\alpha}=\{v\in\Z_q^n\mid v_i=\alpha\}$ for some $i\in[n]$ and $\alpha\in\Z_q$.
This allows one to conclude that each $q$-coloring is a coordinate coloring.
Below we will see that for general $H_q(n,d)$ the maximal independent sets do not have a nice structure.
Moreover, for $d<n$, even if each maximal independent set is the color class of a coordinate coloring, this does not allow one
to conclude that each minimal coloring is a coordinate coloring; see Proposition~\ref{P-qn2} and Example~\ref{E-ColH232}.

For $d<n-q+2$, the upper bound in part~(b) may not be sharp. For instance, for $d=1$ and $n=2q-1$ one obtains
$\lfloor q^{d-1-2r}\sum_{i=0}^r{n-d+1+2r\choose i}(q-1)^i\rfloor=2$, whereas, of course, $\alpha(H_q(n,1))=1$ since $H_q(n,1)$
is the complete graph on $q^n$ vertices.

Now we obtain the following affirmative answers to Question~\ref{Conj-chi}.
It shows that the above method is not very fruitful for the binary case.
Part~(b) below can also be found in \cite[Lem.~18]{RGSS07}.

\begin{cor}\label{C-ConjCases}\
\begin{alphalist}
\item Let $q=2$. Then $\chi(H_2(n,2))=2^{n-1}$ for all $n\in\N_{\geq2}$.
\item Let $q\geq3$. Then $\chi(H_q(n,d))=q^{n-d+1}$ for $d\in\{n-q+2,\ldots,n\}$  (and $d>0$).
      Furthermore, each $q^{n-d+1}$-coloring is even.
\end{alphalist}
\end{cor}

\pagebreak
\begin{proof}
In all given cases one easily verifies the left hand side of~\eqref{e-IndSetChi}.
The last part follows from~\eqref{e-EvenCol}.
\end{proof}

In Example~\ref{E-ColH243} we will see that not all $4$-colorings of $H_2(4,3)$ are even.
An uneven coloring is possible because $\alpha(H_2(4,3))=5>2^{d-1}=4$.
Another uneven $q^{n-d+1}$-coloring is given for $H_3(5,3)$ in Section~\ref{S-OpenProb} showing that the last part of
Corollary~\ref{C-ConjCases}(b) is not true for $d<n-q+2$.

The following case, where $d=2$, is not fully covered by above considerations, but can easily be dealt with in a direct fashion.
\begin{prop}\label{P-qn2}
$\chi(H_q(n,2))=q^{n-1}$ for any~$q$ and each $q^{n-1}$-coloring is even.
\end{prop}
\begin{proof}
We show that each maximal independent set is of the form $\{v+\alpha e_j\mid \alpha\in\Z_q\}$ for some $v\in\Z_q^n$ and some~$j$, and
where~$e_j$ is the $j$-th standard basis vector.
Suppose~$I$ is a maximal independent set containing the two distinct vertices $v,w\in\Z_q^n$.
Then $\dd(v,w)<2$, and thus~$v,\,w$ differ in exactly one position.
Without loss of generality let $v_1\neq w_1$.
But then any other vertex contained in~$I$, say~$z$, satisfies $\dd(z,v)=1=\dd(z,w)$,
and therefore~$z$ also differs from~$v$ and~$w$ only in the first coordinate.
As a consequence, maximality of~$I$ yields $I=\{(\alpha,v_2,\ldots,v_n)\mid \alpha\in\Z_q\}=\{v+\alpha e_1\mid \alpha\in\Z_q\}$,
and thus $|I|=q$.
Now the statements follow from~\eqref{e-IndSetChi} and~\eqref{e-EvenCol}.
\end{proof}

The proof shows that each maximal independent set of $H_q(n,2)$ is the color class of a suitable coordinate coloring (based on $n-1$ coordinates).
However, this does not imply that each $q^{n-1}$-coloring is a coordinate coloring.
We will see this explicitly in Example~\ref{E-ColH232} below.

Since the chromatic number of any graph is always at least as large as the clique number, and cliques in $H_q(n,d)$ are codes
in $\Z_q^n$ of distance~$d$, an affirmative answer to Question~\ref{Conj-chi} can be deduced from the existence of MDS codes.
Recall that a $q$-ary code of length~$n$, cardinality~$M$, and distance~$d$ is called an (unrestricted) $(n,M,d)_q$-code.
The Singleton bound, see~\cite[Thm.~2.4.1]{HP03}, states that for such a code $M\leq q^{n-d+1}$.
Codes attaining this bound are called \emph{MDS codes}, thus they are $(n,q^{n-d+1},d)_q$-codes.
The existence of (unrestricted) MDS codes is a longstanding open problem for most parameters.
We refer to the literature on details about the conjecture; see for
instance~\cite{KKO15} by Kokkala et al.\ and the references therein.
We summarize the relation between existence of MDS codes and the chromatic number as follows.

\begin{rem}\label{R-MDS}
If there exists an $(n,q^{n-d+1},d)_q$-code, then $\chi(H_q(n,d))=q^{n-d+1}$ because
the MDS code in question forms a clique of size~$q^{n-d+1}$ in the Hamming graph.
Again, this is only a sufficient condition for $\chi(H_q(n,d))=q^{n-d+1}$.
For instance, by Theorem~\ref{T-H2nnm1} we have $\chi(H_2(n,n-1))=4$ for all $n$, but there exists no $(n,4,n-1)_2$-code for $n>3$.
The latter follows from the well-known non-existence of nontrivial binary MDS codes, but can also be deduced from the girth in
Proposition~\ref{P-cycles}(c).
Similarly, we have seen already in Corollary~\ref{C-ConjCases}(b) that $\chi(H_3(n,n-1))=9$ for all~$n\geq2$,
but there exists no $(n,9,n-1)_3$-code for $n\geq5$ thanks to the Plotkin bound~\cite[Thm.~2.2.1]{HP03}.
\end{rem}

Of course, for a single Hamming graph, if not too big, one can determine its chromatic number, and actually its colorings, via its
\emph{$r$-coloring ideal} for given~$r$.
For a graph $G=(V,E)$ this ideal is defined as
\begin{equation}\label{e-ColIdeal}
  \cI_r(G):=\big\langle \{x_v^r-1\mid v\in V\}\cup \{{\Ts\sum_{j=0}^{r-1} x_v^jx_w^{r-1-j}}\mid (v,w)\in E\}\big\rangle\subseteq\C[x_v\mid v\in V],
\end{equation}
where we associate independent indeterminates to the vertices of~$G$.
It is well known -- and easy to see -- that $G$ is $r$-colorable if and only if $\cI_r(G)\neq\C[x_v\mid v\in V]$.
In this case the zeros of~$\cI_r(G)$ in~$\C^{|V|}$ correspond to the $r$-colorings of~$G$
(the ideal $\cI_r(G)$ is radical; see Hillar/Windfeldt~\cite[Lem.~3.2]{HiWi08}).
Using this method we obtain
\begin{rem}\label{R-Small}
$\chi(H_2(n,n-2))=8$ for $4\leq n\leq 7$.
\end{rem}

We will briefly return to the $r$-coloring ideal in the next section when considering all $4$-colorings of $H_2(3,2)$.

\section{On the Minimal Colorings of $H_2(n,n-1)$}
In this section we study the $4$-colorings of the graphs~$H_2(n,n-1)$.
Recall from Theorem~\ref{T-H2nnm1} that these are indeed the minimal colorings.
The motivation of this study arose from the following result by Alon et al.~\cite{ADFS04}.
It has been proven already earlier by Greenwell/Lov\'{a}sz in~\cite{GrLo74}, but Alon et al.\ provide an interesting new proof based
on Fourier analysis on the group $\Z_q^n$.

\begin{theo}[\mbox{\cite[Thm.~1.1, Claim~4.1]{ADFS04}}]\label{T-qnn}
Let $q\geq3$. Then $\chi(H_q(n,n))= q$ and every $q$-coloring of $H_q(n,n)$ is a coordinate-coloring.
\end{theo}

In other words, the minimal colorings of $H_q(n,n)$ are exactly the coordinate colorings.
It turns out that this is not the case for the Hamming graphs $H_q(n,n-1)$ so that there are indeed more $q^2$-colorings.
We will show some small examples for $q=2$ later in this section, and in Section~\ref{S-OpenProb} a $9$-coloring of $H_3(4,3)$ is given that is not a
coordinate coloring.

The following notion will be central to our investigation.

\begin{defi}\label{D-TransEdges}
Let $G=(V,E)$ be an undirected simple graph with proper $k$-coloring $K:V\longrightarrow [k]$.
Let $(x,y)\in E$. Then the edge $(x,y)$ is a \emph{transition edge for~$K$} if the map
\[
  K': V\longrightarrow[k],\quad v\longmapsto\left\{\begin{array}{ll} K(x),&\text{if }v=y,\\ K(y),&\text{if }v=x,\\ K(v),&\text{ if }v\not\in\{x,y\}
      \end{array}\right.
\]
is a proper $k$-coloring of~$G$.
In other words, $(x,y)$ is a transition edge if we may swap the colors of~$x$ and~$y$ without destroying properness of the coloring.
In this case, we call~$K'$ the coloring obtained by \emph{swapping colors along the edge $(x,y)$}.
We define
\[
    T(K):=\{(x,y)\mid (x,y)\text{ is a transition edge for }K\}\subseteq E
\]
as the \emph{transition space} of the coloring~$K$.
For any $x\in V$ we also define
\[
   T_x(K):=\{y\in V\mid (x,y)\in T(K)\}
\]
as the set of all neighbors of~$x$ that are adjacent to~$x$ via a transition edge.
\end{defi}
We have the following simple characterization of transition edges.
\begin{rem}\label{R-TransEdges}
Let~$G$ and~$K$ be as in Definition~\ref{D-TransEdges}. Let $(x,y)\in E$. Then
$(x,y)$ is a transition edge for~$K$ iff~$x$ is the only neighbor of~$y$ with color $K(x)$ and~$y$ is the
only neighbor of~$x$ with color $K(y)$.
In other words,
\[
  (x,y)\in T(K)\Longleftrightarrow K^{-1}(K(x))\cap\cN(y)=\{x\}\text{ and }K^{-1}(K(y))\cap\cN(x)=\{y\}.
\]
\end{rem}

The number of transition edges of a given coloring may be regarded as a measure of robustness of that coloring
in the sense that  the coloring tolerates swaps of colors along such edges.
We make this precise with the following notion.

\begin{defi}\label{D-Robust}
Let $G=(V,E)$ be a simple undirected graph.
\begin{alphalist}
\item We define the \emph{robustness} of a $k$-coloring~$K:G\longrightarrow[k]$ as
      $\rb(K)=\frac{|T(K)|}{|E|}$.
\item The \emph{$k$-coloring robustness of~$G$} is defined as
      \[
          \rb_k(G)=\max\{\rb(K)\mid K\text{ is a $k$-coloring of }G\}.
      \]
\end{alphalist}
\end{defi}
Clearly, if $k\geq|V(G)|$ then $\rb_k(G)=1$ because we may color each vertex differently.

Now we are ready to consider some small examples.
\begin{exa}\label{E-ColH232}
Consider the graph $G:=H_2(3,2)$ with vertex set~$\Z_2^3$. 
      \begin{center}
      \resizebox{180pt}{130pt}{\begin{tikzpicture}[scale=0.5]
      \tikzstyle{every node}=[draw,circle,fill=white,minimum size=10pt, inner sep=8pt]
           \draw (0,2) node (000) [label=center:000] {} ;
           \draw (0,0) node (101) [label=center:101] {} ;
           \draw (-3,0) node (011) [label=center:011] {} ;
           \draw (3,0) node (110) [label=center:110] {} ;
           \draw (0,-2) node (010) [label=center:010] {} ;
           \draw (3,-2) node (001) [label=center:001] {} ;
           \draw (-3,-2) node (100) [label=center:100] {} ; 		
           \draw (0,-4) node (111) [label=center:111] {} ;

           \draw (000) to (101);
           \draw (011) to (101);
           \draw (110) to (101);
           \draw (010) to (101);
           \draw (000) to (110);
           \draw (000) to (011);
           \draw (011) to (100);
           \draw (100) to (111);
           \draw (100) to (010);
           \draw (111) to (010);
           \draw (111) to (001);
           \draw (001) to (110);
           \draw (010) to (001);
           \draw (011) to [out=335,in=205](110);
           \draw (100) to [out=335,in=205](001);
           \draw(111) to [out=180, in=180,distance=6cm] (000);
           \end{tikzpicture}}
           \\
           {\footnotesize Figure~1: $H_2(3,2)$}
           \end{center}
      Using, for instance, the $4$-coloring ideal from~\eqref{e-ColIdeal} one obtains the following~$9$ colorings of~$G$.
      These are all colorings up to isomorphism.

      \begin{align*}
        \begin{array}{|c||c|c|c|c|}
        \hline
         K_1&000,\ 001&100,\ 101&010,\ 011&110,\ 111\\ \hline
         K_2&000,\ 010&100,\ 110&001,\ 011&101,\ 111\\ \hline
         K_3&000,\ 100&010,\ 110&001,\ 101&011,\ 111\\ \hline
         K_4&000,\ 100&010,\ 110&001,\ 011&101,\ 111\\ \hline
         K_5&000,\ 100&010,\ 011&001,\ 101&110,\ 111\\ \hline
         K_6&000,\ 010&100,\ 110&001,\ 101&011,\ 111\\ \hline
         K_7&000,\ 010&100,\ 101&001,\ 011&110,\ 111\\ \hline
         K_8&000,\ 001&100,\ 110&010,\ 011&101,\ 111\\ \hline
         K_9&000,\ 001&100,\ 101&010,\ 110&011,\ 111\\ \hline
       \end{array}\\
       \text{\footnotesize Figure~2: All colorings of $H_2(3,2)$\hspace*{1cm}}
      \end{align*}
      The first three rows represent the $(1,2)$-, $(1,3)$-, $(2,3)$-coordinate colorings, respectively.
      The last 6 colorings are obtained from the coordinate colorings by swapping colors along a transition edge.
      For example,~$K_4$ can be obtained from~$K_3$ by swapping the colors of $011$ and $101$ or from~$K_2$ be swapping
      the colors of $010$ and $100$.
      Similarly,~$K_5$ can be obtained from~$K_3$ by swapping the colors of $011$ and $110$ or from~$K_1$ by
      swapping the color of $100$ and $001$.
      In the same way, each of~$K_6,\ldots,K_9$ can be obtained from two different coordinate colorings by swapping colors along a certain transition edge.
      All of this shows that each minimal coloring of~$G$ is either a coordinate coloring or just one swap away from a coordinate coloring.
      In particular, each minimal coloring is even.

      The following two figures display the graph with two different colorings and their transition edges shown in zigzag.
      The left one shows the $(2,3)$-coordinate coloring (coloring~$K_3$) and the right one shows coloring~$K_6$.
      Note that for the $(2,3)$-coordinate coloring we have~$8$ transition edges, and they tile the graph in two $4$-cycles
      (in the sense of Def.~\ref{D-Tiling} later in this paper).
      In contrast, for coloring~$K_6$ we only have~$4$ transition edges.
      Investigating all~$9$ minimal colorings we obtain
      \[
         \rb(K_i)=\left\{\begin{array}{ll}\frac{1}{2},&\text{for }i=1,2,3,\\[.7ex] \frac{1}{4},&\text{for }i=4,\ldots,9.\end{array}\right.
      \]
      \begin{center}
      \resizebox{400pt}{140pt}{\begin{tikzpicture}[scale=0.6]
    \tikzstyle{every node}=[draw,circle,fill=white,minimum size=10pt, inner sep=8pt]

\draw[dashed] (0,2) node (000) [label=center:000] {} ;
\draw[very thick] (0,0) node (101) [label=center:101] {} ;
\draw (-3,0) node (011) [label=center:011] {} ;
\draw[dotted, thick] (3,0) node (110) [label=center:110] {} ;
\draw[dotted, thick] (0,-2) node (010) [label=center:010] {} ;
\draw[very thick] (3,-2) node (001) [label=center:001] {} ;
\draw[dashed] (-3,-2) node (100) [label=center:100] {} ; 		
\draw (0,-4) node (111) [label=center:111] {} ;

\draw[line width=1.8pt] decorate [decoration={zigzag}] {(000) -- (101)};
\draw[line width=1.8pt] decorate [decoration={zigzag}] { (011) to (101)};
\draw (110) to (101);
\draw (010) to (101);
\draw[line width=1.8pt] decorate [decoration={zigzag}] { (000) to (110)};
\draw (000) to (011);
\draw (011) to (100);
\draw (100) to (111);
\draw[line width=1.8pt] decorate [decoration={zigzag}] { (100) to (010)};
\draw[line width=1.8pt] decorate [decoration={zigzag}] { (111) to (010)};
\draw[line width=1.8pt] decorate [decoration={zigzag}] { (111) to (001)};
\draw (001) to (110);
\draw (010) to (001);
\draw[line width=1.8pt] decorate [decoration={zigzag}] { (011) to [out=335,in=205](110)};
\draw[line width=1.8pt] decorate [decoration={zigzag}] {(100) to [out=335,in=205] (001)};
\draw(111) to [out=180, in=180,distance=6cm] (000);

\tikzstyle{every node}=[]
\end{tikzpicture}
  \begin{tikzpicture}[scale=0.6]
    \tikzstyle{every node}=[draw,circle,fill=white,minimum size=10pt, inner sep=8pt]

\draw[dashed] (0,2) node (000) [label=center:000] {} ;
\draw[very thick] (0,0) node (101) [label=center:101] {} ;
\draw[dotted,thick] (-3,0) node (011) [label=center:011] {} ;
\draw (3,0) node (110) [label=center:110] {} ;
\draw[dashed] (0,-2) node (010) [label=center:010] {} ;
\draw[very thick] (3,-2) node (001) [label=center:001] {} ;
\draw (-3,-2) node (100) [label=center:100] {} ; 		
\draw[dotted,thick] (0,-4) node (111) [label=center:111] {} ;

\draw (000) to (101);
\draw[line width=1.8pt] decorate [decoration={zigzag}] { (011) to (101)};
\draw (110) to (101);
\draw (010) to (101);
\draw[line width=1.8pt] decorate [decoration={zigzag}] { (000) to (110)};
\draw (000) to (011);
\draw (011) to (100);
\draw (100) to (111);
\draw[line width=1.8pt] decorate [decoration={zigzag}] { (100) to (010)};
\draw (111) to (010);
\draw[line width=1.8pt] decorate [decoration={zigzag}] { (111) to (001)};
\draw (001) to (110);
\draw (010) to (001);
\draw (011) to [out=335,in=205](110);
\draw (100) to [out=335,in=205](001);
\draw(111) to [out=180, in=180,distance=6cm] (000);

\tikzstyle{every node}=[]
\end{tikzpicture}}
\\
{\footnotesize Fig.~3: $(2,3)$-coordinate coloring with transition edges\hspace*{1cm}Fig.~4: Coloring~$K_6$ with transition edges}
\end{center}
\end{exa}

\begin{exa}\label{E-ColH243}
Consider now the Hamming graph $G=H_2(4,3)$, shown in Figure~5.
In this case one obtains an abundance of $4$-colorings.
We illustrate just a few particular phenomena.
In addition to the 6 coordinate colorings we have for instance the colorings
\begin{align*}
   \begin{array}{|c||p{6.5em}|p{6.5em}|p{6.5em}|p{6.5em}|}
   \hline
   K_1&0000,1000,&0010,1010,&0001,1001,&1110,1101,1011,\\
    &0100,1100 &0110      &0101,0011 &0111,1111\\
    \hline
   K_2&0001, 0010,&1110, 1101,&1111, 1100&0000, 0011,\\
       &0100, 1000 &1011, 0111& 0110, 1010 &0101, 1001\\
   \hline
   \end{array}
\end{align*}
While the first one is an uneven coloring, the second one is even.
Moreover,~$K_2$ is more than one swap away from a coordinate coloring
(of course, the uneven coloring~$K_1$ cannot be obtained by any number of swaps from a coordinate coloring).
Furthermore, one can check straightforwardly that coloring~$K_1$ has~$3$ transition edges,
and coloring~$K_2$ has~$8$ transition edges.
In addition, for~$K_2$ one may move, for instance, $0000$ into the first color class and/or
$1111$ into the second color class to obtain an uneven proper $4$-coloring.
Similarly, moving $0000$ to the third color class of~$K_1$ leads to yet another uneven $4$-coloring.

\begin{center}
    \includegraphics[height=8.5cm]{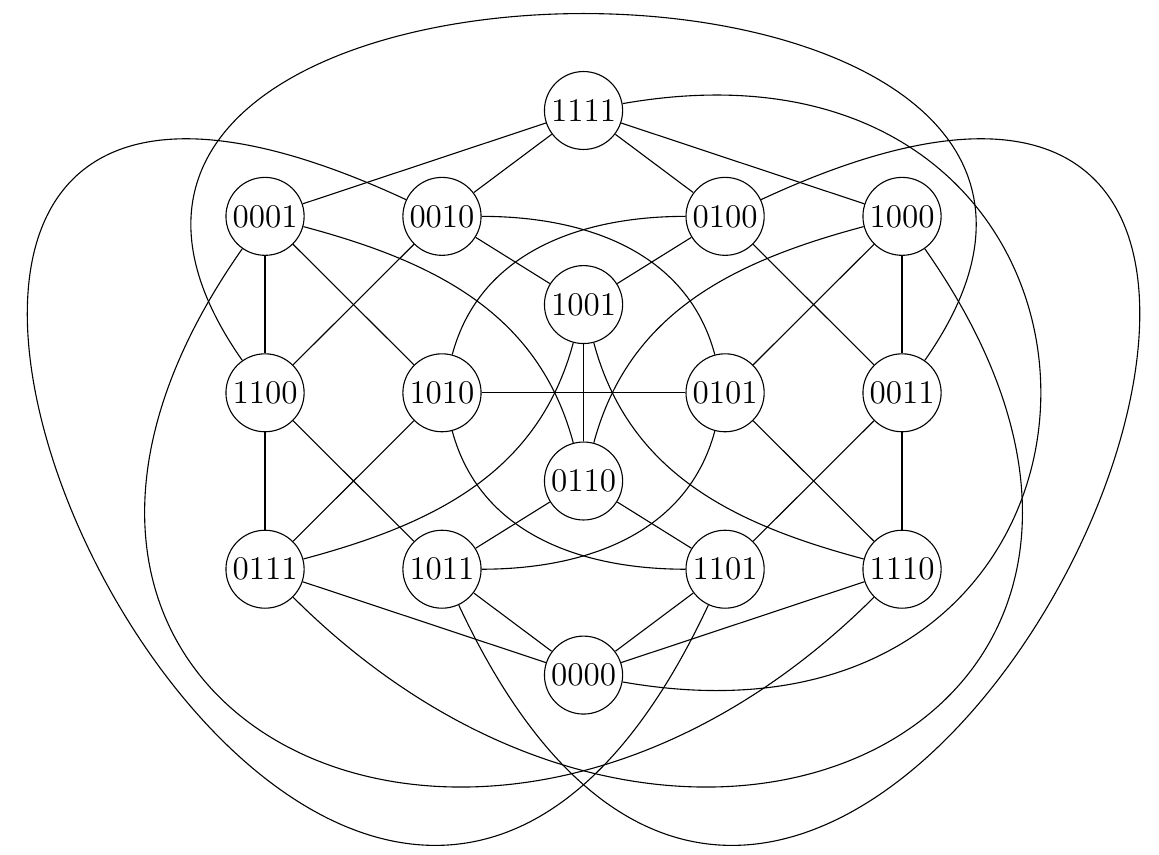}
    \\
    {\footnotesize Figure 5: $H_2(4,3)$}
\end{center}

In Figure~6 we show the $(1,2)$-coordinate coloring along with its transition edges.
As for $H_2(3,2)$ the latter tile the graph in~$4$-cycles (in the sense of Def.~\ref{D-Tiling} later in this paper).
In particular, it has~$16$ transition edges and thus the highest robustness among all colorings discussed in this example.

\begin{center}
    \includegraphics[height=8.5cm]{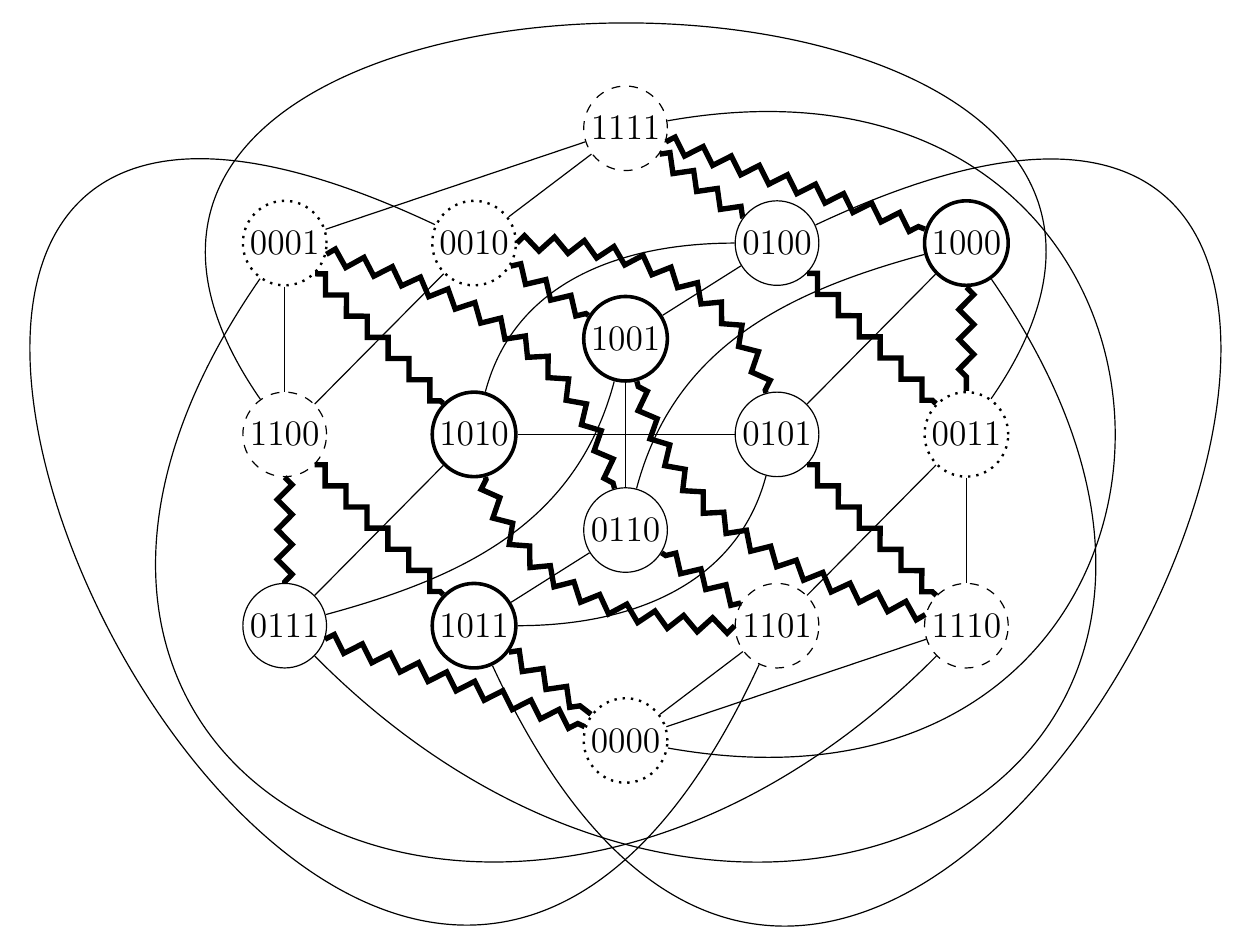}
    \\
    {\footnotesize Figure 6: $(1,2)$-coordinate coloring of $H_2(4,3)$ with transition edges}
\end{center}

\end{exa}

We now return to the study of $4$-colorings of general Hamming graphs $H_2(n,n-1)$.
As the examples above suggest, the coordinate colorings appear to be more robust (that is, have more transition edges)
than other $4$-colorings.
We will see that this is indeed true for all Hamming graphs $H_2(n,n-1)$.
In the next section all maximally robust $4$-colorings will be described explicitly. With one exception they are exactly the coordinate colorings.

Note that for binary Hamming graphs the edges carry a natural label.
Indeed, we may assign the label $x+y=x-y\in\Z_2^n$ to the edge $(x,y)$.
For the following it is helpful to have this edge labeling in mind.
Recall that the Hamming graph $G:=H_2(n,d)$ is the Cayley graph of the group $\Z_q^n$ with generating set
$\cS_d:=\{v\in\Z_2^n\mid \wt(v)\geq d\}$.
In other words, for all $x,y\in\Z_2^n$ we have
\begin{equation}\label{e-additivity}
  x+y\in\cS_d\Longleftrightarrow(x,y)\in E(G) \Longleftrightarrow (x+z,y+z)\in E(G) \text{ for all }z\in\Z_2^n.
\end{equation}

\begin{defi}\label{D-TransGen}
Let~$K$ be a proper coloring of the Hamming graph $H_2(n,d)$, and let $T(K)$ be its transition space.
If there exists a nonempty subset $\cB=\{v_1,\ldots,v_t\}\subseteq\cS_d$ such that
\[
   T(K)=\{(x,y)\mid x+y\in\cB\}
\]
we say that \emph{$T(K)$ is generated by~$\cB$} and write $T(K)=\ideal{\cB}=\ideal{v_1,\ldots,v_t}$.
If no set~$\cB$ exists such that $T(K)=\ideal{\cB}$ or if $T(K)=\emptyset$, we say that \emph{$T(K)$ is not generated}.
\end{defi}
Recall from Definition~\ref{D-TransEdges} the set $T_x(K)$ of transition neighbors of a given vertex~$x$.
With above terminology we have
\begin{equation}\label{e-transx}
    T(K)=\ideal{\cB} \Longleftrightarrow T_x(K)=\{x+b\mid b\in\cB\}=x+\cB \text{ for all }x\in\Z_2^n.
\end{equation}
As this identity indicates, the existence of such a set~$\cB$ is quite a strong assumption.
In particular, every vertex is incident to the same number of transition edges.
In most cases the transition space of a coloring is not generated by a set~$\cB$.

\begin{exa}\label{E-TransGen}\
\begin{alphalist}
\item Figure~3 shows that for the $(2,3)$-coordinate coloring of $H_2(3,2)$ we have $T(K)=\ideal{101,110}$.
      For the coloring~$K_6$ shown in Figure~4, we have $T(K)=\ideal{110}$.
\item From Figure~6 we obtain that for the $(1,2)$-coordinate coloring of $H_2(4,3)$ the transition space is generated by
      $\{0111,1011\}$.
      For the coloring~$K_2$ in Example~\ref{E-ColH243} we have $T(K_2)=\ideal{1111}$.
\item Here is an even coloring of $H_2(5,4)$ for which the transition space is nonempty and not generated.
      Take the $(1,2)$-coordinate coloring and swap the colors of vertices $00000$ and $10111$. This leads to an even
      $4$-coloring, say~$K$ (see also Prop.~\ref{P-TSCoord} below).
      Then one checks straightforwardly that $T_{10111}(K)=\{00000\}$ and $T_{00001}(K)=\{10110,01110\}$.
      This shows that $T(K)$ is not generated.
 \end{alphalist}
\end{exa}

It turns out that the transition space of a coordinate coloring of $H_2(n,n-1)$ is generated and the generating set
can be described explicitly.
To show this we define the following vectors of~$\Z_2^n$. For $i=1,\ldots,n$ let
\[
   e_i:=(0,\ldots,0,1,0,\ldots,0)\ \text{ and }\ f_i:=e_i+\mbone=(1,\ldots,1,0,1,\ldots,1),
\]
where~$1$ resp.~$0$ is at the $i$-th position.
Then for any vertex $x\in\Z_2^n$ its set of neighbors in $H_2(n,n-1)$ is given by
\begin{equation}\label{e-xNeigh}
   \cN(x)=x+\cN(0), \text{ where }\ \cN(0)=\{\mbone,f_1,\ldots,f_n\}=\{v\in\Z_2^n\mid \wt(v)\geq n-1\}.
\end{equation}
Clearly, $\cN(0)=\cS_{n-1}$, with~$\cS_d$ as in~\eqref{e-additivity}.

Before describing the transition spaces of coordinate colorings we first derive an upper bound for the robustness of a
$4$-coloring of $H_2(n,n-1)$.
The case $n=2$ differs from $n\geq 3$.
Indeed, $H_2(2,1)$ is the complete graph $K_4$ and thus each edge is a transition edge.
As a consequence, its robustness is~$1$.
From now on we will only be concerned with the case $n\geq 3$.

\begin{prop}\label{P-TransSpace}
Let $n\geq3$ and $K$ be a proper $4$-coloring of $H_2(n,n-1)$.
Then $|T_x(K)|\leq 2$ for all $x\in\Z_2^n$.
As a consequence, $\rb(K)\leq\frac{2}{n+1}$ with equality if and only if $|T_x(K)|=2$ for all $x\in\Z_2^n$.
\end{prop}
\begin{proof}
Fix any vertex $x\in\Z_2^n$ and consider its set of neighbors $\cN(x)$, which consists of $n+1\geq4$ vertices; see~\eqref{e-xNeigh}.
They can comprise at most~$3$ colors (those different from $K(x)$), and this means that there exist
two neighbors that share a color.
By Remark~\ref{R-TransEdges} only neighbors that do not share a color with another neighbor of~$x$ can be
in $T_x(K)$.
Since there are at most two colors that are not shared by distinct neighbors, we conclude $|T_x(K)|\leq2$ .

For the second part we simply upper bound the number of transition edges.
Since any edge is shared by two vertices we have immediately $|T(K)|\leq\frac{2\cdot2^n}{2}=2^n$
with equality iff $|T_x(K)|=2$ for all $x\in\Z_2^n$.
Moreover, Proposition~\ref{P-basics}(a) yields $|E(H_2(n,n-1))|=(n+1)2^{n-1}$.
Now the statement follows from Definition~\ref{D-Robust}.
\end{proof}

Now we are ready to confirm what we have observed earlier in Example~\ref{E-ColH243}. The coordinate colorings are colorings with
maximal robustness, that is, maximal number of transition edges.
\begin{prop}\label{P-TSCoord}
Let $n\geq3$ and $K$ be the $(i,j)$-coordinate coloring of $H_2(n,n-1)$. Then $T(K)=\ideal{f_i,f_j}$ and thus $\rb(K)=\frac{2}{n+1}$.
As a consequence, $\rb_4(H_2(n,n-1))=\frac{2}{n+1}$.
\end{prop}
\begin{proof}
Without loss of generality let $(i,j)=(1,2)$.
Fix any vertex $x\in\Z_2^n$.
Since $K$ is the $(1,2)$-coordinate coloring, the neighbors $x+\mbone,x+f_3,\ldots,x+f_{n}$ have the same color, say~$A$.
Since $n\geq3$ these are at least two neighbors and therefore none of them can be incident to~$x$ via a transition edge.
Furthermore, $A\neq K(x+f_1)\neq K(x+f_2)\neq A$.
All of this shows that neither neighbor $x+f_1$ and $x+f_2$ shares a color with any other neighbor of~$x$.
Since this is true for all vertices $x\in\Z_2^n$ (and $x=(x+f_i)+f_i$), Remark~\ref{R-TransEdges} yields that $(x,x+f_1)$ and $(x,x+f_2)$ are
the transition edges incident to~$x$.
Now~\eqref{e-transx} and Proposition~\ref{P-TransSpace} conclude the proof.
\end{proof}

\begin{rem}\label{R-H2ndRobust}
The previous result can be generalized to the coordinate colorings of general binary Hamming graphs $H_2(n,d)$.
Precisely, let $K$ be the $(c_1,\ldots,c_{n-d+1})$-coordinate coloring of~$H_2(n,d)$.
Then one straightforwardly shows that
\[
  T(K)=\Big\langle\Big\{\mbone+\sum_{j=1\atop j\neq i}^{n-d+1} e_{c_j}\,\Big|\, 1\leq i\leq n-d+1\Big\}\Big\rangle.
\]
This leads to the robustness $\rb(K)=\frac{n-d+1}{\sum_{i=0}^{n-d}{n \choose i}}$.
It is not clear whether this is the maximum robustness among all $2^{n-d+1}$-colorings of~$G$.
\end{rem}

In the next section we will describe explicitly the maximally robust $4$-colorings of $H_2(n,n-1)$.

\section{Maximally Robust \mbox{$4$}-Colorings of \mbox{$H_2(n,n-1)$}}
A characterization of the maximally robust $4$-colorings of $H_2(n,n-1)$ is derived in terms of their
transition spaces and all these colorings are presented explicitly (up to isomorphism).
In order to formulate these results, we need the following terminology.

\begin{defi}\label{D-Tiling}
Let~$G=(V,E)$ be a finite graph and $M\subseteq E$ be a set of edges in~$G$.
We say that \emph{$M$ tiles the graph in $4$-cycles} if~$M$ is the union of pairwise vertex-disjoint and edge-disjoint
$4$-cycles and every vertex appears on (exactly) one cycle.
\end{defi}
For example, in Figures~3 and~6 the transition edges of the given coloring tile the graph in $4$-cycles.

Now we are ready to present our main results.
The first theorem provides a characterization of maximally robust colorings in terms of the transition spaces.
The second theorem presents for each such transition space the unique coloring associated with it or shows that no
such coloring exists.
The proofs will be presented afterwards.

Recall from the paragraph right before Proposition~\ref{P-TransSpace} that the case $n=2$ is trivial.

\begin{theo}\label{T-MainThm}
Let $n\geq3$ and~$K$ be a $4$-coloring of $H_2(n,n-1)$.
The following are equivalent.
      \begin{romanlist}
      \item There exist $v,w\in\cN(0)$ such that $T(K)=\ideal{v,w}$.
      \item $\rb(K)=\frac{2}{n+1}$.
      \item The transition edges tile the graph in $4$-cycles.
      \end{romanlist}
 If one, hence any of the above is true, then the coloring of a single $4$-cycle uniquely determines the coloring of the entire graph.
\end{theo}
Recall that if the transition space of a coloring is generated by a set~$\cB$, then
$\cB$ is contained in~$\cN(0)=\{\mbone,f_1,\ldots,f_n\}$.
Therefore the requirement  in i) that the vertices $v,\,w$ be in $\cN(0)$ is not a restriction.

\begin{theo}\label{T-MainThm2}
Let $n\geq3$ and~$K$ be a $4$-coloring of $H_2(n,n-1)$.
\begin{arabiclist}
\item If $T(K)=\ideal{f_i,f_j}$ for some $1\leq i<j\leq n$, then $K$ is the $(i,j)$-coordinate coloring.
\item If $T(K)=\ideal{\mbone,f_j}$ for some~$j$, then~$n$ is even and~$K$ is the coloring with color sets
      \[
         \cA^{\nu}_{\mu}:=\{v\in\Z_2^n\mid \wt(v)\equiv\nu\!\!\!\!\mod2,\ v_j=\mu\}\text{ for }\nu,\,\mu\in\{0,1\}.
      \]
\end{arabiclist}
As a consequence, any $4$-coloring of $H_2(n,n-1)$ with maximal robustness is even.
\end{theo}
Note that Part~2) tells us that if~$n$ is odd, then sets of the form $\{\mbone,f_j\}$ do not generate the
transition space of a $4$-coloring.

We prove the results in several steps and first need some preliminary results.
Before doing so, let us point out that for any $4$-coloring the graph $H_2(n,n-1)$ does not have cycles consisting
of transition edges of length~$3$.
For $n\geq4$ this is obvious by Proposition~\ref{P-cycles}(c), while for $n=3$ this can be verified via the Figures~3 and~4,
which cover all $4$-colorings of $H_2(3,2)$ up to graph isomorphism.
The following lemmas will lead to information about the existence, colors, and interrelation of $4$-cycles consisting of transition edges.

\begin{lemma}\label{L-H2n4}
Consider the graph $G:=H_2(n,n-1)$, where $n\geq3$. Let $x_1,x_2,x_3,x_4\in V(G)$ be distinct vertices such that
\begin{equation}\label{e-cycleai}
   (x_1,x_2),\, (x_2,x_3),\,(x_3,x_4),\,(x_4,x_1)\in E(G),
\end{equation}
that is, $x_1,\ldots,x_4$ are on a $4$-cycle. Then $x_1+x_2=x_3+x_4$ and thus $x_1+x_4=x_2+x_3$.
\end{lemma}
\begin{proof}
As all vectors are in $\Z_2^n$ it suffices to show $x_1+x_2=x_3+x_4$.
Since $x_1,\ldots,x_4$ are distinct, we have
\[
  x_1+x_2\neq x_2+x_3\neq x_3+x_4.
\]
By~\eqref{e-cycleai} and~\eqref{e-additivity} all these vectors are in $\cN(0)=\{\mbone,f_1,\ldots,f_n\}$.
Using that $x_1+x_4=(x_1+x_2)+(x_2+x_3)+(x_3+x_4)$, we have $\wt\big((x_1+x_2)+(x_2+x_3)+(x_3+x_4)\big)=\wt(x_1+x_4)\geq n-1$.
But since the sum of any three distinct vectors from $\cN(0)$ has weight at most~$n-2$, we conclude
$x_1+x_2=x_3+x_4$.
\end{proof}

\begin{lemma}\label{L-4Coloring}
Let $n\geq3$ and let~$K$ be a $4$-coloring of $H_2(n,n-1)$.
\begin{alphalist}
\item Suppose there exists a path of length~$3$ consisting of transition edges and passing through the
      vertices $x_1,x_2,x_3,x_4$.
      Then the colors of~$x_1,x_2,x_3,x_4$ are distinct.
\item There exists no simple path in~$H_2(n,n-1)$ of length greater than~$3$ consisting of transition edges.
\end{alphalist}
\end{lemma}
\begin{proof}
Both claims are easily verified if $n=3$ as there are only~$9$ colorings to check; see Figure~2.
Indeed, Figures~3 and~4 show (up to graph isomorphism) the only possible scenarios of transition edges.
Thus we may assume $n\geq4$.
\\
(a)  Since $(x_1,x_2),(x_2,x_3),(x_3,x_4)$ are transition edges, the only vertices that may share a color are~$x_1$ and~$x_4$.
Pick a vector $y\in\cN(0)\,\backslash\,\{x_1+x_2,\,x_2+x_3,\,x_3+x_4\}$. This is possible because $n\geq3$.
By~\eqref{e-additivity} the vertices $x_2+y$ and $x_3+y$ are adjacent.
Since the girth is~$4$ thanks to Proposition~\ref{P-cycles}(c), the choice of~$y$ implies that the vertices $x_1,x_2,x_3,x_4,x_2+y,x_3+y$ are distinct.
Thus we have the subgraph
\begin{center}
\begin{tikzpicture}[scale=0.45]
    \draw (0,2) node (x) [label=center:$x_1$] {} ;
    \draw (5,2) node (v) [label=center:$x_2$] {} ;
    \draw (10,2) node (w) [label=center:$x_3$] {} ;
    \draw (15,2) node (z) [label=center:$x_4$] {} ;
    \draw (5,0) node (vy) [label=center:\raisebox{-2ex}{$x_2+y$}] {} ;
    \draw (10,0) node (wy) [label=center:\raisebox{-2ex}{$x_3+y$}] {} ;
    \draw (6,0) node(a1) {};
    \draw (9,0) node (a2) {};

    \draw (x) to (v);
    \draw (v) to (w);
    \draw (w) to (z);
    \draw (vy) to (v);
    \draw (wy) to (w);
    \draw (a1) to (a2);

\end{tikzpicture}
\end{center}
Since $(x_1,x_2)$ and $(x_2,x_3)$ are transition edges, we conclude $K(x_2+y)=K(x_4)$.
In the same way $K(x_3+y)=K(x_1)$.
Now the adjacency of $x_2+y$ and $x_3+y$ implies $K(x_1)\neq K(x_4)$.
All of this shows that $x_1,x_2,x_3,x_4$ assume distinct colors.
\\
(b) Suppose there exists a path consisting of transition edges passing through the vertices
$x_1,\ldots,x_5$, and where $x_i\neq x_j$ for $i\neq j$.
From Part~(a) we know that $x_1,x_2,x_3,x_4$ have distinct colors and the same is true for $x_2,x_3,x_4,x_5$.
Thus we conclude $K(x_1)=K(x_5)$.
Consider the vertex $x_2+x_3+x_4$, which is adjacent to $x_2$ and $x_4$ by~\eqref{e-additivity}. Thus we have
\begin{center}
\begin{tikzpicture}[scale=0.45]
    \draw (0,2) node (x) [label=center:$x_1$] {} ;
    \draw (5,2) node (v) [label=center:$x_2$] {} ;
    \draw (10,2) node (w) [label=center:$x_3$] {} ;
    \draw (15,2) node (z) [label=center:$x_4$] {} ;
    \draw (20,2) node (zz) [label=center:$x_5$] {} ;
    \draw (10,0) node (vy) [label=center:\raisebox{-2ex}{$x_2+x_3+x_4$}] {} ;
    \draw (11,0) node(a1) {};
    \draw (9,0) node (a2) {};

    \draw (x) to (v);
    \draw (v) to (w);
    \draw (w) to (z);
    \draw (z) to (zz);
    \draw (a1) to (z);
    \draw (a2) to (v);
    \end{tikzpicture}
\end{center}
The fact that the upper row consists of transition edges implies that $x_2+x_3+x_4$ must equal $x_1$ or $x_5$ for otherwise there is
no color left for $x_2+x_3+x_4$.
Assume without loss of generality that $x_2+x_3+x_4=x_1$. But then $K(x_5)=K(x_1)=K(x_2+x_3+x_4)$, contradicting the fact
that $(x_4,x_5)$ is a transition edge.
Thus we have shown that no such path exists.
\end{proof}

By the previous lemma there are no cycles of length greater than~$4$ in $H_2(n,n-1)$ consisting of transition edges.
The next lemma provides information about cycles of length~$4$ consisting of transition edges.

\begin{lemma}\label{L-4cycle}
Let $n\geq3$ and let~$K$ be a $4$-coloring of $G:=H_2(n,n-1)$.
Let $x_1,x_2,x_3,x_4\in V(G)$ be distinct vertices such that
\[
   (x_1,x_2),\, (x_2,x_3),\,(x_3,x_4),\,(x_4,x_1)\in T(K),
\]
that is, $x_1,\ldots,x_4$ are on a $4$-cycle whose edges are transition edges of~$K$.
Then
\begin{alphalist}
\item The colors $K(x_1),\ldots,K(x_4)$ are distinct.
\item Let $y\in V(G)\backslash\{x_1,\ldots,x_4\}$ be any other vertex of~$G$. Then~$y$ is adjacent to at most one of the vertices
      $x_1,\ldots,x_4$.
      Furthermore, if~$y$ is adjacent to~$x_i$, then $K(y)=K(\tilde{x_i})$, where $\tilde{x_i}$
      denotes the vertex opposite to~$x_i$ in the $4$-cycle.
\end{alphalist}
\end{lemma}

\begin{proof}
Consider the given $4$-cycle with its opposite vertices
\begin{equation}\label{e-4cycle}
\text{\begin{minipage}{5cm}
\begin{tikzpicture}[scale=0.45]
    \draw (4,0) node (a3) [label=center:\raisebox{-.7ex}{$x_3$}] {} ;
    \draw (0,2) node (a2) [label=center:$x_2\;$] {} ;
    \draw (8,2) node (a4) [label=center:$\;x_4$] {} ;
    \draw (4,4) node (a1) [label=center:\raisebox{.7ex}{$x_1$}] {} ;
    \draw (a3) to (a2);
    \draw (a1) to (a4);
    \draw (a3) to (a4);
    \draw (a1) to (a2);
\end{tikzpicture}
\end{minipage}
\qquad\begin{minipage}{4cm} $\tilde{x_1}=x_3$\\[.7ex] $\tilde{x_3}=x_1$\\[.7ex] $\tilde{x_2}=x_4$\\[.7ex] $\tilde{x_4}=x_2$\end{minipage}}
\end{equation}
(a) follows from Lemma~\ref{L-4Coloring}(a).
\\
(b) Let~$y\in V(G)\backslash\{x_1,\ldots,x_4\}$.
Assume $x_1\sim y\sim x_2$. Then $K(x_1)\neq K(y)\neq K(x_2)$.
But since $(x_1,x_4)$ and $(x_2,x_3)$ are transition edges, we also have $K(x_4)\neq K(y)\neq K(x_3)$.
This contradicts that~$K$ is a $4$-coloring. In the same way one shows that~$y$ cannot be adjacent to two opposite vertices
of the $4$-cycle.
Assume now $y\sim x_1$. Then, again, because $(x_1,x_2)$ and $(x_1,x_4)$ are transition edges,~$y$ must have color $K(x_3)=K(\tilde{x_1})$.
\end{proof}

Now we are ready to prove the main theorems.

\noindent\emph{Proof of Theorem~\ref{T-MainThm}.}
i)~$\Rightarrow$~ii) follows from Proposition~\ref{P-TransSpace}.
\\
ii)~$\Rightarrow$~iii)
Example~\ref{E-ColH232} tells us that the only colorings of $H_2(3,2)$ with maximal robustness are
the coordinate colorings. Therefore the implication is easily verified for $n=3$; see Figure~3.
Thus let $n\geq4$ and hence $\text{girth}(H_2(n,n-1))=4$ due to Proposition~\ref{P-cycles}(c).
Thanks to Proposition~\ref{P-TransSpace} we have $|T_x(K)|=2$ for all $x\in\Z_2^n$, that is, every vertex is incident
to exactly two transition edges.
The finiteness of the graph then implies that every vertex is on a unique cycle consisting of transition edges, and
due to Lemma~\ref{L-4Coloring}(b) all these cycles have length~$4$.
\\
iii)~$\Rightarrow$~i)
Let $(0,v),(v,z),(z,w),(w,0)$ be the $4$-cycle consisting of transition edges containing the vertex~$0$.
From Lemma~\ref{L-H2n4} we obtain $v=w+z$, thus $z=v+w$.
This shows that $T_x(K)=x+\{v,w\}$ for $x\in\{0,v,w,z\}$.
Note also that $\wt(v),\,\wt(w)\geq n-1$.
Let now $x\in V(H_2(n,n-1))\,\backslash\{0,v,w,z\}$ be any other vertex. We have to show that $T_x(K)=x+\{v,w\}$.
We consider two cases.
\\
\underline{Case 1:} $x$ is adjacent to one of the vertices $\{0,v,w,z\}$. In this case it is adjacent to exactly one of these
vertices thanks to Lemma~\ref{L-4cycle}(b).
Without loss of generality we may assume $x\sim0$, which then implies $\wt(x)\geq n-1$.
Then by~\eqref{e-additivity}
\begin{equation}\label{e-shiftcycle}
  (x,x+v),\,(x+v,x+z),\,(x+z,x+w),\,(x+w,x)
\end{equation}
is a $4$-cycle in $H_2(n,n-1)$.
It can be thought of as a shift by~$x$ of the original cycle.
Again by~\eqref{e-additivity} we obtain that $x+y\sim y$ for $y\in\{0,v,w,z\}$.
Thus we have the subgraph
\begin{center}
\begin{tikzpicture}[scale=0.45]
    \draw (4,0) node (z) [label=center:\raisebox{-.7ex}{$z$}] {} ;
    \draw (0,2) node (v) [label=center:$v\;$] {} ;
    \draw (8,2) node (w) [label=center:$\;w$] {} ;
    \draw (4,4) node (0) [label=center:\raisebox{.7ex}{$0$}] {} ;
    \draw (z) to (v);
    \draw (0) to (w);
    \draw (z) to (w);
    \draw (0) to (v);

    \draw (17,0) node (xz) [label=center:\raisebox{-.7ex}{$x+z$}] {} ;
    \draw (13,2) node (xv) [label=center:$x+v\;$] {} ;
    \draw (21,2) node (xw) [label=center:$\;x+w$] {} ;
    \draw (17,4) node (x) [label=center:\raisebox{.7ex}{$x$}] {} ;
    \draw (xz) to (xv);
    \draw (x) to (xw);
    \draw (xz) to (xw);
    \draw (x) to (xv);
    \draw (0) to [out=20,in=160] (x);
    \draw (v) to [out=20,in=160] (xv);
    \draw (w) to [out=20,in=160] (xw);
    \draw (z) to [out=20,in=160] (xz);

\end{tikzpicture}
\end{center}
Now Lemma~\ref{L-4cycle}(b) tells us that
\begin{equation}\label{e-colorsC}
    K(x)=K(z),\ K(x+v)=K(w),\ K(x+z)=K(0),\ K(x+w)=K(v).
\end{equation}
Using that $x\sim0$ we see that~$x$ has neighbors of the three distinct colors $K(0),\,K(w),\,K(v)$.
Now we can identify the transition edges incident to~$x$.
The fact that $T_0(K)=\{v,w\}$ tells us that $(x,0)$ is not a transition edge.
But then~$x$ cannot be incident to any transition edge $(x,y)$ such that $K(y)=K(0)$.
Since $K(v)$ and $K(w)$ are neighboring colors of~$x$ and by assumption~$x$ is on a $4$-cycle consisting of transition edges, we conclude that
$(x,x+v)$ and $(x,x+w)$ are the transition edges incident to~$x$.
Thus $T_x(K)=x+\{v,w\}$.
By generality of~$x$ all edges in the cycle~\eqref{e-shiftcycle} are transition edges.
Note also that the coloring of the cycle~\eqref{e-shiftcycle} is uniquely determined by~\eqref{e-colorsC}.
\\
\underline{Case~2:} $x$ is not adjacent to any vertex of the original $4$-cycle.
As we have shown in Proposition~\ref{P-basics}(e) the graph $H_2(n,n-1)$ is connected and therefore there exists a path
from~$0$ to~$x$. Assume the vertices of this path are $0,y_1,\ldots,y_t=x$.
With the aid of this path, we can now translate the original cycle across the graph and obtain new $4$-cycles consisting
of transition edges and whose colors are uniquely determined by the coloring of the original $4$-cycle.
Case~1 guarantees that in each step we obtain $T_{y_i}(K)=y_i+\{v,w\}$.
Thus we conclude $T_x(K)=x+\{v,w\}$, as desired.
This part also showed that the coloring~$K$ is uniquely determined by a single $4$-cycle,  and thus the proof is complete.
\mbox{}\hfill$\Box$

\medskip
\noindent\emph{Proof of Theorem~\ref{T-MainThm2}.}
1) From Theorem~\ref{T-MainThm} we know that there exists (up to isomorphism) at most one $4$-coloring~$K$ such that
$T(K)=\ideal{f_i,f_j}$.
Now the result follows from Proposition~\ref{P-TSCoord}.
\\[1ex]
2) In what follows the parity of the Hamming weight will play a central role, and so we introduce
$\wt_2(x):=\wt(x)\!\!\!\mod2\in\Z_2$ for any $x\in\Z_2^n$ and thus compute with these weights modulo~$2$.
Recall that in~$\Z_2^n$ we have $\wt(x+y)=\wt(x)+\wt(y)-2\big|\{i\mid x_i=1=y_i\}\big|$, and therefore
\begin{equation}\label{e-wtparity}
   \wt_2(x+y)=\wt_2(x)+\wt_2(y)\text{ for all }x,\,y\in\Z_2^n.
\end{equation}
We will also compute modulo~$2$ with the indices of the set~$\cA^{\nu}_{\mu}$ introduced in the theorem.

\noindent\underline{Case 1:} Assume~$n$ is even. By Theorem~\ref{T-MainThm} and Lemma~\ref{L-4cycle}(a) there is at most one coloring~$K$ such that
$T(K)=\ideal{\mbone,f_j}$.
Thus it suffices to verify that the given~$K$ satisfies the statements.
First of all, it is clear that the sets $\cA^{\nu}_{\mu}$ are pairwise disjoint and partition the vertex set~$\Z_2^n$.
Next, since $n$ is even, we have $\wt_2(\mbone)=0$ and $\wt_2(f_i)=1$ for all $i=1,\ldots,n$.
This allows us to describe the behavior of the sets~$\cA^{\nu}_{\mu}$ along edges of the graph.
Indeed,~\eqref{e-wtparity} leads to
\begin{equation}\label{e-shiftA}
  x+\cA^{\nu}_{\mu}=
  \left\{\begin{array}{ll}\cA^{\nu}_{\mu+1},&\text{if }x=\mbone,\\[.6ex]
         \cA^{\nu+1}_{\mu},&\text{if }x=f_j,\\[.6ex]
         \cA^{\nu+1}_{\mu+1},&\text{if }x=f_i\text{ for }i\neq j\end{array}\right.
\end{equation}
for all $\nu,\,\mu\in\Z_2$.
This shows that the sets $\cA^{\nu}_{\mu}$ are independent and thus form the color classes of a proper $4$-coloring, say~$K$.
It remains to show that $T_x(K)=x+\{\mbone,f_j\}$ for all $x\in\Z_2^n$.
Due to~\eqref{e-shiftA} the vectors~$\mbone$ and~$f_j$ induce $4$-cycles in $H_2(n,n-1)$ of the form
\begin{center}
\begin{tikzpicture}[scale=0.45]
    \draw (4,0) node (a3) [label=center:\raisebox{-1ex}{$\cA^0_0$}] {} ;
    \draw (0,2) node (a2) [label=center:$\cA^1_0\;$] {} ;
    \draw (8,2) node (a4) [label=center:$\;\cA^0_1$] {} ;
    \draw (4,4) node (a1) [label=center:\raisebox{.7ex}{$\cA^1_1$}] {} ;
    \draw (3.8,0.2) node (a31) {};
    \draw (4.2,0.2) node (a32) {};
    \draw (4.2,3.8) node (a12) {};
    \draw (3.8,3.8) node (a11) {};
    \draw (0.2,2) node (a22) {};
    \draw (7.8,2) node (a42) {};
    \draw (a31) to (a22);
    \draw (a12) to (a42);
    \draw (a32) to (a42);
    \draw (a11) to (a22);
    \draw (1.6,0.6) node (b) [label=center:$f_j$] {} ;
    \draw (6,0.6) node (c) [label=center:$\mbone$] {} ;
    \draw (1.6,3.4) node (d) [label=center:$\mbone$] {} ;
    \draw (6,3.6) node (e) [label=center:$f_j$] {} ;
\end{tikzpicture}
\end{center}
Furthermore, the last case in~\eqref{e-shiftA} tells us that the only other neighbors of a vertex in~$\cA^{\nu}_{\mu}$ are in
$\cA^{\nu+1}_{\mu+1}$ and thus all have the same color.
Since this color  is different from the neighboring colors in the cycle, the edges in these $4$-cycles are indeed all the transition
edges of~$K$, that is, $T(K)=\ideal{\mbone,f_j}$.
Identity~\eqref{e-shiftA} shows that the sets $\cA^{\nu}_{\mu}$ all have the same cardinality and hence~$K$ is an even coloring.

\noindent\underline{Case 2:} Assume~$n$ is odd. We have to show that there is no proper $4$-coloring with
transition space~$\ideal{\mbone,f_j}$. Without loss of generality let $j=n$.
It will be beneficial to work again with the vertex partition $\Z_2^n=\cA^0_0\cup\cA^0_1\cup\cA^1_0\cup\cA^1_1$
(which are not color classes at this point).
As~$n$ is odd we have for all $\nu,\,\mu\in\Z_2$
\begin{equation}\label{e-shiftB}
  x+\cA^{\nu}_{\mu}=
  \left\{\begin{array}{ll}\cA^{\nu+1}_{\mu+1},&\text{if }x=\mbone,\\[.6ex]
         \cA^{\nu}_{\mu},&\text{if }x=f_n,\\[.6ex]
         \cA^{\nu}_{\mu+1},&\text{if }x=f_i\text{ for }i<n\end{array}\right.
\end{equation}
Thus $\ideal{\mbone,f_j}$ induces $4$-cycles with vertices in the sets $\cA^{\nu}_{\mu}$ as follows
\begin{equation}\label{e-two4cycles}
\begin{array}{ll}
\begin{tikzpicture}[scale=0.45]
    \draw (4,0) node (a3) [label=center:\raisebox{-1ex}{$\cA^1_1$}] {} ;
    \draw (0,2) node (a2) [label=center:$\cA^1_1\;$] {} ;
    \draw (8,2) node (a4) [label=center:$\;\cA^0_0$] {} ;
    \draw (4,4) node (a1) [label=center:\raisebox{.7ex}{$\cA^0_0$}] {} ;
    \draw (3.8,0.2) node (a31) {};
    \draw (4.2,0.2) node (a32) {};
    \draw (4.2,3.8) node (a12) {};
    \draw (3.8,3.8) node (a11) {};
    \draw (0.2,2) node (a22) {};
    \draw (7.8,2) node (a42) {};
    \draw (a31) to (a22);
    \draw (a12) to (a42);
    \draw (a32) to (a42);
    \draw (a11) to (a22);
    \draw (1.6,0.6) node (b) [label=center:$f_n$] {} ;
    \draw (6,0.6) node (c) [label=center:$\mbone$] {} ;
    \draw (1.6,3.4) node (d) [label=center:$\mbone$] {} ;
    \draw (6,3.6) node (e) [label=center:$f_n$] {} ;
\end{tikzpicture}&
\hspace*{1cm}
\begin{tikzpicture}[scale=0.45]
    \draw (4,0) node (a3) [label=center:\raisebox{-1ex}{$\cA^1_0$}] {} ;
    \draw (0,2) node (a2) [label=center:$\cA^1_0\;$] {} ;
    \draw (8,2) node (a4) [label=center:$\;\cA^0_1$] {} ;
    \draw (4,4) node (a1) [label=center:\raisebox{.7ex}{$\cA^0_1$}] {} ;
    \draw (3.8,0.2) node (a31) {};
    \draw (4.2,0.2) node (a32) {};
    \draw (4.2,3.8) node (a12) {};
    \draw (3.8,3.8) node (a11) {};
    \draw (0.2,2) node (a22) {};
    \draw (7.8,2) node (a42) {};
    \draw (a31) to (a22);
    \draw (a12) to (a42);
    \draw (a32) to (a42);
    \draw (a11) to (a22);
    \draw (1.6,0.6) node (b) [label=center:$f_n$] {} ;
    \draw (6,0.6) node (c) [label=center:$\mbone$] {} ;
    \draw (1.6,3.4) node (d) [label=center:$\mbone$] {} ;
    \draw (6,3.6) node (e) [label=center:$f_n$] {} ;
\end{tikzpicture}
\end{array}
\end{equation}
Assume now that~$K$ is a proper $4$-coloring such that $T(K)=\ideal{\mbone,f_n}$.
Then the above $4$-cycles consist of transition edges.
Lemma~\ref{L-4cycle} yields that the vertices of any such cycle attain distinct colors.
Denote the $4$ colors by $A,B,C,D$.
Using $v=0\in\cA^0_0$ as an instance of the left cycle, we have without loss of generality
\begin{equation}\label{e-inicolors}
  K(0)=A,\ K(\mbone)=B,\ K(e_n)=C,\ K(f_n)=D.
\end{equation}
Theorem~\ref{T-MainThm} tells us that this determines uniquely the coloring of each other transition $4$-cycle (and thus of every vertex),
and Lemma~\ref{L-4cycle} specifies how this is done.
In order to describe this precisely, the following sets will be helpful.
For $\nu,\,\mu\in\Z_2$ and $i=0,\ldots,n$ define
\[
   B^{\nu}_{\mu}(i):=\{v\in\cA^{\nu}_{\mu}\mid \wt(v)=i\}.
\]
Clearly, $B^{\nu}_{\mu}(i)=\emptyset$ if $i\not\equiv\nu\!\!\!\mod2$.
One easily verifies
\begin{equation}\label{e-Bcycle}
  \mbone+B^{\nu}_{\mu}(i)=B^{\nu+1}_{\mu+1}(n-i),\quad
  f_n+B^{\nu}_0(i)=B^{\nu}_0(n-1-i),\quad
  f_n+B^{\nu}_1(i)=B^{\nu}_1(n+1-i).
\end{equation}
This shows that the vertices in the $4$-cycles in \eqref{e-two4cycles} specify to
\begin{equation}\label{e-twoBcycles}
\begin{array}{ll}
\begin{tikzpicture}[scale=0.45]
    \draw (4,-.4) node (a3) [label=center:$B^1_1(i+1)$] {} ;
    \draw (-1.5,2) node (a2) [label=center:$B^1_1(n-i)$] {} ;
    \draw (10.2,2) node (a4) [label=center:$\;B^0_0(n-1-i)$] {} ;
    \draw (4,4.3) node (a1) [label=center:$B^0_0(i)$] {} ;
    \draw (3.8,0.2) node (a31) {};
    \draw (4.2,0.2) node (a32) {};
    \draw (4.2,3.8) node (a12) {};
    \draw (3.8,3.8) node (a11) {};
    \draw (0.2,2) node (a22) {};
    \draw (7.8,2) node (a42) {};
    \draw (a31) to (a22);
    \draw (a12) to (a42);
    \draw (a32) to (a42);
    \draw (a11) to (a22);
    \draw (1.4,0.8) node (b) [label=center:$f_n$] {} ;
    \draw (6.3,0.8) node (c) [label=center:$\mbone$] {} ;
    \draw (1.4,3.2) node (d) [label=center:$\mbone$] {} ;
    \draw (6.3,3.4) node (e) [label=center:$f_n$] {} ;
\end{tikzpicture}&
\begin{tikzpicture}[scale=0.45]
    \draw (4,-.4) node (a3) [label=center:$B^1_0(n-1-i)$] {} ;
    \draw (-.7,2) node (a2) [label=center:$B^1_0(i)$] {} ;
    \draw (9.5,2) node (a4) [label=center:$B^0_1(i+1)$] {} ;
    \draw (4,4.3) node (a1) [label=center:$B^0_1(n-i)$] {} ;
    \draw (3.8,0.2) node (a31) {};
    \draw (4.2,0.2) node (a32) {};
    \draw (4.2,3.8) node (a12) {};
    \draw (3.8,3.8) node (a11) {};
    \draw (0.2,2) node (a22) {};
    \draw (7.8,2) node (a42) {};
    \draw (a31) to (a22);
    \draw (a12) to (a42);
    \draw (a32) to (a42);
    \draw (a11) to (a22);
    \draw (1.4,0.8) node (b) [label=center:$f_n$] {} ;
    \draw (6.3,0.8) node (c) [label=center:$\mbone$] {} ;
    \draw (1.4,3.2) node (d) [label=center:$\mbone$] {} ;
    \draw (6.3,3.4) node (e) [label=center:$f_n$] {} ;
\end{tikzpicture}
\end{array}
\end{equation}
The cycles on the left (resp.\ right) hand side exist only if~$i$ is even (resp.\ odd).
Let $n=2t+1$.
We claim that the sets $B^{\nu}_{\mu}(i)$ attain the following colors:
\begin{equation}\label{e-colors}
\left.\begin{array}{llll}
  K(B^0_0(i))=A,& K(B^1_0(n-1-i))=A,\qquad\\[.7ex]
  K(B^1_1(n-i))=B,& K(B^0_1(i+1))=B,\\[.7ex]
  K(B^1_1(i+1))=C,& K(B^0_1(n-i))=C,\\[.7ex]
  K(B^0_0(n-1-i))=D,\qquad& K(B^1_0(i))=D,
\end{array}\right\} \text{ for all }0\leq i\leq t.
\end{equation}
If we can prove this, we arrive at a contradiction because $B^0_0(t)=B^0_0(n-1-t)$.

In order to prove~\eqref{e-colors} we induct on~$i$.
For $i=0$ we have
\[
    B^0_0(0)=\{0\},\ B^1_1(n)=\{\mbone\},\ B^1_1(1)=\{e_n\},\ B^0_0(n-1)=\{f_n\},
\]
and~\eqref{e-inicolors} establishes the left hand column of~\eqref{e-colors}.
The sets in the right hand column are all empty for $i=0$.

Let us now assume~\eqref{e-colors} for all $i<i'$.
Suppose~$i'$ is even and let $v\in B^0_0(i')$.
Let $\alpha\in\{1,\ldots,n-1\}$ be such that $v_{\alpha}=1$.
Then $w:=v+f_{\alpha}\in B^0_1(n+1-i')=B^0_1(n-(i'-1))$.
This means that~$w$ is the top vertex of a cycle as on the right hand side in~\eqref{e-twoBcycles}, whereas~$v$
is the top vertex of a cycle as on the left hand side.
These two cycles are obtained by a shift with~$f_{\alpha}$.
By assumption we know that the vertices of the right hand cycle assumes the colors $C,D,A,B$ counterclockwise with $K(w)=C$.
Now Lemma~\ref{L-4cycle} establishes the colors of the left hand cycle as in the first column of~\eqref{e-colors}.
The same reasoning applies to the case where~$j$ is odd, in which one starts with an arbitrary element in $B^1_0(i')$.
This concludes the proof.
\mbox{}\hfill$\Box$

\section{Open Problems}\label{S-OpenProb}
In this section we briefly address further questions pertaining to the coloring of the Hamming graphs.

\begin{que} First of all, it remains to answer Question~\ref{Conj-chi} for the cases not covered in Section~\ref{S-chi}.
\end{que}

\begin{que}
Can one describe the maximally robust $2^{n-d+1}$-colorings of $H_2(n,d)$?
Are their transition spaces generated?
Remark~\ref{R-H2ndRobust} is but a first step in this direction.
\end{que}

\begin{que}
From Corollary~\ref{C-ConjCases} we know that $\chi(H_q(n,n-1))=q^2$ for all $q\geq3$ and all $q^2$-colorings are even.
Can one derive further information about these minimal colorings?
\end{que}
Not surprisingly, just as in the binary case not all minimal colorings are coordinate colorings.
Here is a $9$-coloring of $H_3(4,3)$ that is not a coordinate coloring.
Consider the independent set $\cI:=\{x\in\Z_3^4\mid \wt(x)\leq1\}$
and the linear MDS code (see Remark~\ref{R-MDS} and the paragraph preceeding it)
\begin{equation}\label{e-MDS}
   \cC=\text{rowspace}\begin{pmatrix}0&1&1&1\\1&0&1&2\end{pmatrix}:=\{a(0111)+b(1012)\mid a,b\in\Z_3\}\subseteq\Z_3^4.
\end{equation}
Then each of the shifts $v+\cI,\,v\in\cC$, is clearly an independent set, and these shifts are pairwise disjoint.
Thus they form the color classes of a $9$-coloring that is not a coordinate coloring.
Note that~$\cI$ is the set of coset leaders of the code~$\cC$
(that is, each $v\in\cI$ is the unique vector of smallest weight in the coset~$v+\cC$, and these cosets partition~$\Z_3^4$).
In fact,~$\cC$ is a perfect code, see~\cite[Sec.~1.12]{HP03}.
This example thus generalizes whenever we have a linear perfect MDS code over a finite field.
The latter are exactly the Hamming codes with parameter $r=2$, hence the
$q$-ary Hamming codes of length~$n=q+1$, dimension $q-1$, and distance~$d=3$ and where~$q$ is a prime power and the alphabet
the finite field~$\F_q$ of order~$q$; see for instance \cite[Thm.~1.12.3]{HP03}.
As a consequence, we obtain a $q^{q-1}$-coloring of $H_q(q+1,3)$ that is not a coordinate coloring.
As before, the color classes are the shifts of the independent set $\{x\in\F_q^{q+1}\mid \wt(x)\leq1\}$.

Incidentally, the example above can be utilized to create an uneven 27-coloring of $H_3(5,3)$
(recall however, that it is not clear whether this graph has chromatic number~$27$).
Using~$\cC$ as in~\eqref{e-MDS} we define the sets
\begin{align*}
  \hat{\cC}&:=\{(v_1,\ldots,v_4,0)\mid (v_1,\ldots,v_4)\in\cC\}\subseteq\Z_3^5,\\
  \cJ&:=\{v\in\Z_3^5\mid \wt(v)\leq 1\},\\
  \cB_\alpha&:=\{(v_1,\ldots,v_4,\alpha)\mid \wt(v_1,\ldots,v_4)=1\} \text{ for }\alpha=1,2.
\end{align*}
Then~$\cJ$ is an independent set of cardinality~$11$ (showing that the upper bound in Theorem~\ref{T-alpha}(b) is
attained for $H_3(5,3)$), and~$\cB_1$ and~$\cB_2$ are independent sets of size~$8$.
It is straightforward to see that the 27 sets $v+\cJ,\,v+\cB_1,\,v+\cB_2$, where $v\in\hat{\cC}$, form a partition of~$\Z_3^5$
and thus serve as the color classes of an uneven $27$-coloring.

\begin{que}
Do the $q^{n-d+1}$-colorings of $H_q(n,d)$ for $q\geq3$ have transition edges?
\end{que}

For $d=1$ the graph is the complete graph
on $q^n$ vertices and thus $T(K)=E(H_q(n,1))$ for every $q^n$-coloring.
Furthermore, for $q\geq3,\,d\geq2$ it is easy to see that the $q^{n-d+1}$-coordinate colorings of $H_q(n,d)$
have no transition edges.
Indeed, consider the coloring based on the coordinates $(i_1,\ldots,i_{n-d+1})$.
Let $(x,y)\in E(H_q(n,d))$, thus $\wt(x-y)\geq d$.
Choose a coordinate~$j\in[n]\,\backslash\,\{i_1,\ldots,i_{n-d+1}\}$.
Then at least one of $x+e_j$ and $x+2e_j$ is a neighbor of~$y$ because at least one of them disagrees with~$y$ in the $j$-th coordinate and thus
has at least distance~$d$ from~$y$.
Since both these vertices agree with~$x$ in the positions $i_1,\ldots,i_{n-d+1}$, they have the same color as~$x$.
This shows that $(x,y)$ is not a transition edge.

Furthermore, the even colorings obtained from MDS Hamming codes as described above do not have transition edges either.
This can be shown by a similar argument as in the previous paragraph.
Finally, it is straightforward to verify that the uneven $27$-coloring of $H_3(5,3)$ given above has no transition edges.

Since for the binary Hamming graphs $H_2(n,n-1)$ the coordinate colorings have the maximum possible number of transition edges (that is, maximum robustness), we
conjecture that no $q^{n-d+1}$-coloring of $H_q(n,d),\,q\geq3$, has any transition edges.

We close with the following question which arises from the considerations of maximally robust colorings of $H_2(n,n-1)$.
\begin{que}
Suppose $K$ is a $q^{n-d+1}$-coloring of $H_q(n,d)$ such that its transition space $T(K)$ is generated. Does this imply that~$K$ is even?
\end{que}

\bibliographystyle{abbrv}

\end{document}